\documentclass[11pt,reqno]{amsart}

\usepackage{amssymb}
\usepackage{amsmath}
\usepackage{amsfonts}
\usepackage{graphicx}
\usepackage{tikz}
\usepackage[width=150mm,height=220mm,
centering,
]{geometry}
\usepackage{cite}
\usepackage{algorithm,algorithmic}
\allowdisplaybreaks
\def\supp{\mathrm{supp}\,}

\begin{document}
\newtheorem{theorem}{Theorem}[section]
\newtheorem{lemma}[theorem]{Lemma}
\newtheorem{definition}[theorem]{Definition}
\newtheorem{example}[theorem]{Example}
\newtheorem{corollary}[theorem]{Corollary}
\newtheorem{remark}[theorem]{Remark}
\newtheorem{proposition}[theorem]{Proposition}
\numberwithin{equation}{section}
\def\arg{\mathrm{arg}}
\def\rmd{\mathrm d}

\bibliographystyle{abbrv}

\title
{STFT Phase Retrieval with Two Window Functions}
\author[T. Chen]{Ting Chen}
\address{School of Mathematical Sciences and LPMC,
Nankai University,
Tianjin,
China}
\email{t.chen@nankai.edu.cn}

\author[H. Lu]{Hanwen Lu}
\address{School of Mathematical Sciences and LPMC,
Nankai University,
Tianjin,
China}
\email{2013537@mail.nankai.edu.cn}

\author[W. Sun]{Wenchang Sun}
\address{School of Mathematical Sciences and LPMC,
Nankai University,
Tianjin,
China}
\email{sunwch@nankai.edu.cn}

\author[Y. Zhao]{Yutong Zhao}
\address{School of Mathematical Sciences and LPMC,
Nankai University,
Tianjin,
China}
\email{zhaoyt@mail.nankai.edu.cn}

\date{}


\thanks{Corresponding author: Wenchang Sun.}

\thanks{This work was supported by the National Natural Science Foundation of China (grant No. 12271267 and 12571104) and the Fundamental Research Funds for the Central Universities.}

\begin{abstract}
In this paper, we consider the uniqueness of STFT phase retrieval with two window functions.
We show that a complex-valued locally integrable
nonseparable signal is uniquely determined
up to a global phase by phaseless samples
of its short time Fourier transforms
with respect to two well-chosen window
functions
over countable parallel lines
or certain lattices.
Moreover, we give the optimal sampling interval
for STFT phase retrieval with compactly supported
window functions.
For periodic locally integrable signals,
we obtain a uniqueness result for STFT phase
retrieval with sampled values
over two parallel lines
whose distance is an irrational multiple of the period.
And for quasi-periodic signals, we obtain
a similar result.
\end{abstract}

\maketitle

\section{Introduction}

In many fields of modern sciences, due to the fact that detectors can only record the modulus of a signal and lose phase information, how to recover a signal from its absolute values of linear measurements has attracted widespread attention. Note that it is impossible to distinguish between $f$ and $\lambda f$ based only on the phaseless  measurements of $f$, where
\[
\lambda \in \mathbb{T}:=\{z \in \mathbb{C}:|z|=1\}.
\]
Therefore, the phase retrieval problem can be accurately described as follows:
can a signal $f$ be recovered from its absolute values
of linear measurements up to a global phase. This problem is widely used in a large number of applications, such as signal processing \cite{2006signal,2014signal,2020signal,2019signal,
2014signal2,
Sign2021},
diffraction imaging \cite{2007Diffractive,
2016Alternating,
2008PDiffractive,
2002X-ray}, optics \cite{1978object,1988optics,2015Optical} and quantum information theory \cite{2013Quantum,2015quantum}.

Phase retrieval is studied
in various aspects with different linear measurements, which include phaseless sampling \cite{2019signal,2011unsigned,2024entire},
Fourier transforms \cite{1956fourier,1957fourier,1987fourier,
2020fourier,
2014pauli}, wavelet transforms \cite{2023wavelet,2015wavelet,2017wavelet},
and short-time Fourier transforms
\cite{2021stft,2022gabor,2018stft,
2022stft,2025stft,2025stft1,
2019gabor,
2021stft1,
2024stlct,
2004entire}.

Recall that for a signal $f \in L^{2}(\mathbb{R})$, its
short time Fourier transform (STFT) with respect to
a window function $\phi\in L^{2}(\mathbb{R})$ is defined by
\begin{align}
\mathcal{V}_{\phi} f(t, \omega):=\int_{\mathbb{R}} f(x) \overline{\phi(x-t)} e^{-2 \pi i x \omega} \rmd x, \quad (t, \omega) \in \mathbb{R}^2. \label{eq:stft}
\end{align}
If a signal $f$ is
uniquely determined by its phaseless STFT measurements $|V_{\phi}f|$
up to a global phase,
then we say that the window function $\phi$ does
STFT phase retrieval for the signal $f$.
If $\phi$ does STFT phase retrieval for
all signals in a function space,
then we say that $\phi$ does STFT phase retrieval for
the function space.

The ambiguity function
\cite{Foundations2001} is a useful tool in the study
of STFT phase retrieval.
It was shown in \cite{2021stft,2020fourier,2019gabor}
that when $f,
\phi\in L^2(\mathbb R)$, then
\[
  \mathcal F (|\mathcal V_{\phi}f|^2)(\eta,\xi)
    = c_{\phi} \mathcal V_ff(-\xi,\eta)
    \overline{\mathcal V_{\phi}\phi (-\xi,\eta)}.
\]
When the ambiguity function $\mathcal A\phi(t,\omega):= e^{\pi i t\omega}\mathcal V_{\phi}\phi(t,\omega)$ of  window $\phi$ is nowhere
vanishing,
$\phi$ does STFT phase retrieval for
$L^2(\mathbb R)$
~\cite{2021stft,2019gabor}.
By analysing
ambiguity functions $\mathcal A\phi$
and
$\mathcal A f $,
Bartusel \cite{Bartusel2023}
proved that a large class of window functions
do phase retrieval for connected $L^2$ signals.

In this paper, we focus on STFT phase retrieval with compactly supported integrable window functions $\phi$.
That is,
there is a positive number $B$ such that $\supp
\phi \subset [-B,B]$ and $\phi\in L^1[-B,B]$.

For the STFT of a signal $f$ with respect to a
compactly supported window function
to make sense, it suffices for $f$ to
be in $L_{loc}^1(\mathbb R)$
(see Lemma~\ref{Lm:L0}),
that is, $f$ is integrable on
any interval of finite length.
Although compactly supported window functions
belong to the window class studied
in \cite{Bartusel2023},
we have to study the STFT phase retrieval
with new method since the ambiguity function
$\mathcal A f$
does not make sense when $f\in L_{loc}^{1}\setminus L^1$. In fact, our analysis is based on a result
of local STFT phase retrieval (see Lemma~\ref{lem1}).

On the other hand,
since the window function is compactly supported,
it does not do STFT phase retrieval
for separable functions.
To be specific, we introduce the following definition.

\begin{definition}
We call a function $L$-separable if there exists an interval of length $L$ such that $f$ vanishes on this interval.

If there is no such an interval, we call $f$
$L$-nonseparable.
\end{definition}

For STFT phase retrieval with sampled values,
Alaifari
and Wellershoff \cite{2022gabor}
showed that the Gaussian window function
with any lattice does not do STFT phase retrieval.
And Grohs and Liehr~\cite{GrohsLiehr2023}
proved the same conclusion for arbitrary
window functions.
That is,
for any $g\in L^2$ and $a,b>0$,
there is some $f\in L^2(\mathbb R)$ which
is not uniquely determined up to
a global phase by $\{|\mathcal V_{g}f(ma,nb)|\}_{m,n\in\mathbb Z}$.
And in \cite{2025stft1},
Grohs, Liehr and Rathmair
considered the STFT phase retrieval problem with four window functions and established the uniqueness theorem on a lattice for square integrable
functions, where each window function is a linear combination of two given Hermite functions.

In this paper, we study
STFT phase retrieval
with two
compactly supported
window functions and
sampled values on certain lattices,
for which we give the optimal sampling interval.
Moreover, our results are valid for  signals
which are locally integrable.
As a special case, we obtain the following.

\begin{theorem}\label{thm:main}
Suppose that $B$ is a positive contant,
$\phi \in  L^{p'}(\mathbb{R})$ for some $1\le p\le \infty$
with $\operatorname{supp}( \phi) \subset
[-B, B]$, $\overline{\phi(-x)}=\phi(x)$ and
\begin{align*}
\phi(x) \ne 0,\quad \mbox{ a.e. } x\in [-B, B].
\end{align*}
Let $\psi(x)=\phi(x) (e^{ 2\pi ixb}-1) $, where $0<b\le 1/(2B)$.

Then for any $0<a\le B$ and $f\in L^p(\mathbb R)$ which is $(2B-a)$-nonseparable, $f$ is determined up to a global phase by the measurements
\begin{align} \label{eq:ed2}
\{|\mathcal{V}_{\phi} f(ma,n/(4B))|,~|\mathcal{V}_{\psi} f(ma,n/(4B))|:n,m \in \mathbb{Z}\}.
\end{align}

Moreover, the conclusion fails if $a>B$ or $f$ is $(2B-a)$-separable.
\end{theorem}

The paper is organized as follows. In Section 2, we recall some fundamental results
on Fourier transforms and properties of
holomorphic functions
which are used in the proofs.
In Section 3, we study STFT phase retrieval
from semi-discrete samples.
We establish a uniqueness result
for STFT phase retrieval of locally integrable
functions.
And in Section 4, we study STFT phase retrieval
from discrete samples.
We consider three types of
signals which include periodic,
quasi-periodic, and locally integrable signals.
For all types of signals we show
that STFT phase retrieval holds for nonseparable signals.
Specifically, for periodic
and quasi-periodic  locally integrable signals,
we obtain uniqueness results for STFT phase retrieval with sampled values
over two parallel lines
whose distance is an irrational multiple of the period.
And for general locally integrable signals,
we obtain uniqueness results with
sampled values over countable lines.

\section{Preliminaries}

In this section, we collect some preliminary results to be used in the proofs of main results.

\subsection{Some notations}
For a signal $f \in L^{1}(\mathbb{R})$, its Fourier transform is defined as
\begin{align*}
\mathcal{F} f(\omega)=\hat f(\omega)
  :=\int_{\mathbb{R}} f(x) e^{-2 \pi i x \omega} \rmd x, \quad \omega \in \mathbb{R}.
\end{align*}
The Fourier transform extends to a unitary operator on $L^{2}(\mathbb{R})$ by a density argument.

For $\tau, \nu\in \mathbb{R}$, the translation
operator $T_{\tau}$
and modulation operator $M_{\nu}$ are defined respectively by
\begin{align*}
T_{\tau} f(x)= f(x-\tau).
\end{align*}
and
\begin{align*}
M_{\nu} f(x)=e^{2 \pi i \nu x} f(x).
\end{align*}

The conjugate reflection of a function $f$ is defined by
\[
   f^*(x) = \overline{f(-x)}.
\]

For two functions $f$ and $g$,
$f\sim g$ stands for $f = \lambda g$ for some
constant $\lambda\in\mathbb T$.

\subsection{Completeness of complex exponentials}
\label{sec:sub2}
Let $f\in L^1[-B,B]$ for some $B>0$. The Fourier transform of $f$ is
\[
  \hat f(\omega) = \int_{-B}^B f(t) e^{-2\pi i t\omega} \rmd t,\quad
  \omega\in\mathbb R.
\]
$\hat f$ extends to a function defined on the complex plane $\mathbb C$
with
\[
  \hat f(z) = \int_{-B}^B f(t) e^{-2\pi i tz} \rmd t,\quad
  z\in\mathbb C.
\]
Since $f\in L^1[-B,B]$,
it is easy to see that $\hat f$ is continuous on $\mathbb C$. Moreover,
since $ e^{-2\pi i tz}$ is holomorphic as a function of $z$,
by Fubini's theorem and Cauchy's theorem, the integral of
$f$ along any closed path is $0$. Hence $\hat f$ is holomorphic on $\mathbb C$, thanks to Morera's theorem. Furthermore,
\[
  |\hat f(z)|\le C e^{2\pi B|z|}.
\]
Hence $\hat f$ extends to an exponential type entire function with type $2\pi B$.

Observe that
\begin{align}
  |\hat f(\omega)|^2
  &=
  \int_{-B}^B\int_{-B}^B f(t) \overline{f(s)} e^{-2\pi i (t-s)\omega} \rmd t
   \, \rmd s \nonumber \\
  &=
  \int_{-2B}^{2B}\Big(\int_{[-B-t,B-t]\cap[-B,B]} f(t+s) \overline{f(s)} \rmd s \Big) e^{-2\pi i t\omega} \rmd t
   .  \label{eq:ea3}
\end{align}
Similar arguments show that
$|\hat f|^2$ extends to an exponential type entire function with type $4\pi B$.


It is well known that a function integrable on $[-B,B]$
is uniquely determined by its Fourier coefficients.
For a proof, see
\cite[Corollary 1.11]{Duoandikoetxea2001}.

\begin{proposition}\label{prop:p3}
If $f\in L^1[-B,B]$ and
\[
  \int_{-B}^B f(x) e^{- \pi i nx/B  } \rmd x = 0,\quad \forall n\in\mathbb Z,
\]
then $f=0$, a.e.
\end{proposition}

When the sampling interval decreases slightly, only ``half'' Fourier coefficients
are needed to determine an integrable function.
Specifically,
the system
\[
    \{e^{\pi in x/B}\}_{n\ge 1}
\]
is complete in $C[-A,A]$ if $0<A<B$, which is a consequence
of Carleman's theorem~\cite[Page 116]{Young2001}.

Moreover, the restriction on sampling points can be further relaxed.

\begin{proposition}
\label{prop:p2}
Let $A>0$ and $\{\omega_n\}_{n\ge 1}$ be a sequence of positive
numbers such that
\[
  \limsup_{n\to\infty}  \frac{n}{\omega_n} > {2A}.
\]
Then the system $\{e^{2\pi i\omega_n x}\}_{n\ge 1}$ is complete
in $C[-A,A]$.

Consequently, if $f\in L^1[-A,A]$ and
\[
  \int_{-A}^A f(x) e^{2\pi i \omega_n x} \rmd x =0, \quad \forall n\ge 1,
\]
then $f=0$, a.e.
\end{proposition}

\begin{proof}
The first part is a consequence
of the Levinson theorem~\cite[Page 138]{Young2001}.

For the second part, consider the functional
\[
  T:\, \phi\mapsto \int_{-A}^A f(x) \phi(x) \rmd x.
\]
Since $f$ is integrable, $T$ is continuous on $C[-A,A]$.
The hypotheses state  that
$T(e^{2\pi i \omega_n x})=0$ for all $n\ge 1$.
Since the  system $\{e^{2\pi i\omega_n x}\}_{n\ge 1}$ is complete
in $C[-A,A]$, we have $T=0$,
which implies that $f=0$, a.e.
\end{proof}

\subsection{Phase retrieval of compactly supported functions from their Fourier transforms}

As shown in Section~\ref{sec:sub2},
the Fourier transform of a compactly supported integrable
function extends to an entire function of exponential
type.
In some cases,
$f$ and $\overline{f}$ have the same phaseless measurements,
which leads to the study of conjugate phase retrieval \cite{2004entire,2020Conjugate,2021Conjugate,2024Conjugate}.
To distinguish $f$ and $\overline f$, we need more measurements.
McDonald \cite{2004entire} gave some sufficient conditions for phase retrieval of finite order entire functions.
As a consequence, the uniqueness theorem for complex-valued bandlimited functions was established.

The following result is a variant of
\cite[Theorem 1]{2021Conjugate}
and
\cite[Theorem 3]{2004entire}.
Since we have slightly different hypotheses on
signals, to make the paper more readable, we include
a proof here.

\begin{lemma}\label{Lm:L2}
Let $B>0$ be a constant and
$\{\omega_n\}_{n\ge 1} = \{k/(4B)\}_{k\in\mathbb Z}$
or   $\{\omega_n\}_{n\ge 1}$ be a sequence of positive numbers such that
\begin{equation}\label{eq:ea1}
  \limsup_{n\to\infty} \frac{n}{\omega_n} > 4B.
\end{equation}
Suppose that $f,g\in L^1[-B,B]$, $0<b\le 1/(2B)$.
If
\begin{equation}\label{eq:ea2}
 |\hat f(\omega_n)| = |\hat g(\omega_n)|,
 \quad
 |\hat f(\omega_n)-\hat f(\omega_n+b)|
   = |\hat g(\omega_n)-\hat g(\omega_n+b)|,  \quad n\ge 1,
\end{equation}
then either $f\sim g$ or $f\sim g^* $.
\end{lemma}

\begin{proof}
Fix  $f$ and $g$.
We see from (\ref{eq:ea3}) that
$H(\omega):=|\hat f(\omega)|^2 - |\hat g(\omega)|^2$
extends to an entire function of the form
\[
  H(z) = \int_{-2B}^{2B} h(t) e^{-2\pi i t z}\rmd t,
\]
where $h\in L^1[-2B,2B]$. It follows from (\ref{eq:ea2}) that
\begin{equation}\label{eq:ea4}
  H(\omega_n) = 0.
\end{equation}
Applying Proposition~\ref{prop:p3}
for $\{\omega_n\}_{n\ge 1} = \{k/(4B)\}_{k\in\mathbb Z}$ or applying
Proposition~\ref{prop:p2} for $\{\omega_n\}_{n\ge 1}$ meeting (\ref{eq:ea1}),
we obtain that
$h(t)=0$, a.e. on $[-2B,2B]$.
Consequently,
\[
  |\hat f(\omega)| = |\hat g(\omega)|,\quad \forall \omega\in\mathbb R.
\]
Similarly,
\[
   |\hat f(\omega)-\hat f(\omega+b)|
   = |\hat g(\omega)-\hat g(\omega+b)|, \quad  \forall \omega\in\mathbb R.
\]
Now, it follows from \cite[Theorem 1]{2004entire}
that either $\hat f = W\hat g$ or $\hat f=W \bar {\hat g}$, where $W$ is a continuous function
with period $b$.

If $\hat f=W\hat g$, then
\[
  \hat f(nb) = W(0) \hat g(nb),\quad \forall n\in\mathbb Z.
\]
It follows from Proposition~\ref{prop:p3} that
$\hat f = W(0) \hat g$ and therefore
$f =  W(0) g$ .

If $\hat f=W \bar {\hat g}$,
similar arguments show that
$  f =W(0) g^*$.
\end{proof}

\subsection{Locally periodic functions}
We call a function $f$ locally periodic
with a period $T$
on an interval $I$
if $f(x)=f(x+T)$ for almost all $x\in I$ with $x+T\in I$.

The following result is known. Since we do not find a reference,
we include a proof here.

\begin{lemma}\label{Lm:L1}
Let $\{a_k\}_{k\ge 1} $ be a sequence of positive numbers
such that
$\lim_{k\to\infty} a_k=0$.
Suppose that $f$ is locally integrable
and locally periodic on some interval
$[x_0,x_1]$
with periods $a_k$, $k\ge 1$. That is,
\[
   f(x) = f(x+a_k),\quad \mbox{ a.e. } x\in [x_0,x_1-a_k], \ \forall k\ge 1.
\]
Then $f(x)=c_0$, a.e. on $[x_0, x_1]$
for some constant $c_0$.
\end{lemma}

\begin{proof}
Let $x,y$ be Lebesgue points of $f$ in $(x_0,x_1)$. Since a periodic function has equal integrals over any interval whose length equals its period, for $a_k$ small enough,
we have
\[
  \frac{1}{a_k}  \int_x^{x+a_k}  f(t)\rmd t
  =\frac{1}{a_k}  \int_y^{y+a_k}  f(t)\rmd t.
\]
Letting $k\to\infty$, we obtain that
\[
  f(x) = f(y).
\]
Since $f$ is locally integrable, almost every $x\in(x_0,x_1)$ is a Lebesgue point of $f$.
Fix a Lebesgue point $x_0$. Then we have
\[
  f(x) = f(x_0),\quad a.e.
\]
\end{proof}

\section{Phase retrieval from semi-discrete samples}

In this section, we consider the uniqueness of STFT
phase retrieval.
Since the window functions are compactly supported,
we need only some local properties of signals.
Specifically, we consider signals which are locally
integrable.

The following lemma shows that
the STFT of a locally integrable function
is well defined.

\begin{lemma}\label{Lm:L0}
Let $B>0$  and $\phi \in  L^1(\mathbb{R})$ with $\operatorname{supp}( \phi) \subset
[-B, B]$.
Suppose that $f\in L_{loc}^1(\mathbb R)$. Then
for
almost all $t$ in $\mathbb R$,
$f \phi(\cdot-t)$ is integrable on $\mathbb R$.
\end{lemma}

\begin{proof}
For any $a<b$, we see from Fubini's theorem that
\begin{align*}
  \int_a^b \int_{\mathbb R} |f(x)\phi(x-t)| \rmd x\, \rmd t
&=  \int_a^b \int_{|x-t|\le B} |f(x)\phi(x-t)| \rmd x\, \rmd t \\
&\le  \int_{a-B}^{b+B} \int_{|x-t|\le B} |f(x)\phi(x-t)| \rmd t\, \rmd x \\
&=\|\phi\|_{L^1} \|f\|_{L^1[a-B,b+B]}<\infty.
\end{align*}
Hence for almost all $t\in [a,b]$,
$f\phi(\cdot-t)\in L^1$.
Since $a,b$ are arbitrary,
$f\phi(\cdot-t)$ is integrable for almost all $t
\in\mathbb R$.
\end{proof}

Before giving a result on STFT phase retrieval
of locally integrable functions, we present
a result on local STFT phase retrieval,
which is useful in our study.

Let $B>0$ be a constant.
Suppose that the window function $\phi$ satisfies
\begin{align}  \label{eq:phi}
\supp  \phi \subset
[-B, B],\quad \phi^*(x)=\phi(x)
\mbox{\quad and \quad }
\phi(x) \ne 0, \mbox{ a.e. } x\in [-B, B].
\end{align}

\begin{lemma}\label{lem1}
Let $\phi \in  L^1(\mathbb{R})$ meet  (\ref{eq:phi}).
Suppose that $\{\omega_n\}_{n\ge 1} = \{k/(4B)\}_{k\in\mathbb Z}$
or $\{\omega_n\}_{n\ge 1}$ is a sequence of positive numbers
satisfying (\ref{eq:ea1}).
Let $f, g\in L^1_{loc}(\mathbb R)$. Then
for almost all $t\in\mathbb R$,
both $f \phi(\cdot-t)$ and $g\phi(\cdot -t)$ are integrable on $\mathbb R$.
Fix such a number $t_{0}$. Then the following two assertions
 are equivalent.

\begin{enumerate}
\item There is some $\lambda_0\in\mathbb T$ such that
either
\begin{equation}\label{eq:ea5}
f(x) = \lambda_0 g(x),\quad \mbox{ a.e. \, on } \, [t_0-B , t_0+B]
\end{equation}
or
\begin{equation}\label{eq:ea6}
  f(x) = \lambda_0 \overline{g(-x+2t_{0})},\quad \mbox{ a.e. \, on } \, [t_0-B , t_0+B].
\end{equation}

\item

\begin{align}\label{equ1}
\left\{\begin{array}{l}
|\mathcal{V}_{\phi} f(t_{0},\omega_{n})|=|\mathcal{V}_{\phi} g(t_{0},\omega_{n})|, \\
|\mathcal{V}_{\psi} f(t_{0},\omega_{n})|=|\mathcal{V}_{\psi} g(t_{0},\omega_{n})|,
\end{array}\right.
 \quad \forall n \ge 1,
\end{align}
where $\psi(x)=\phi(x) (e^{ 2\pi ixb}-1) $, $0<b\le 1/(2B)$.
\end{enumerate}
\end{lemma}

\begin{proof}
By Lemma~\ref{Lm:L0}, for almost all $t\in\mathbb R$,
both $f \phi(\cdot-t)$ and $g\phi(\cdot -t)$ are integrable on $\mathbb R$.

(i) $\Rightarrow$ (ii). Note that if (\ref{eq:ea5}) holds, then
\begin{align*}
|\mathcal{V}_{\phi} f(t_{0},\omega_{n})|=|\mathcal{V}_{\phi} g(t_{0},\omega_{n})|,~
|\mathcal{V}_{\psi} f(t_{0},\omega_{n})|=|\mathcal{V}_{\psi} g(t_{0},\omega_{n})|,\quad \forall n\ge 1.
\end{align*}
On the other hand, assume that (\ref{eq:ea6}) holds. By the definition of STFT, we calculate that
\begin{align}\label{equ3}
\mathcal{V}_{\phi} f(t_{0}, \omega)
&=\int_{\mathbb{R}} f(x) \overline{\phi(x-t_{0})} e^{-2 \pi i x \omega} \rmd x\nonumber \\
&=e^{-2 \pi i t_{0} \omega }\int_{-B}^{B}f(x+t_{0}) \overline{\phi(x)} e^{-2 \pi i x \omega} \rmd x.
\end{align}
Then together with (\ref{eq:ea6}), we have
\begin{align*}
\left|\mathcal{V}_{\phi} f(t_{0}, \omega)\right|&=\left|\int_{-B}^{B}f(x+t_{0}) \overline{\phi(x)} e^{-2 \pi i x \omega} \rmd x\right|\\&=\left|\int_{-B}^{B}\overline{g(-x+t_{0})} \overline{\phi(x)} e^{-2 \pi i x \omega} \rmd x\right|\\&=\left| \int_{-B}^{B}g(x+t_{0}) \phi(-x) e^{-2 \pi i x \omega} \rmd x\right|.
\end{align*}
Since $\overline{\phi(-x)}=\phi(x)$, we obtain that
\begin{align*}
\left|\mathcal{V}_{\phi} f(t_{0}, \omega)\right|&=\left|\int_{-B}^{B}g(x+t_{0}) \overline{\phi(x)} e^{-2 \pi i x \omega} \rmd x\right|\\
&=\left|\mathcal{V}_{\phi} g(t_{0}, \omega)\right|.
\end{align*}
Observe that $\overline{\psi(-x)}=\psi(x)$. Similar arguments show that
\begin{align*}
\left|\mathcal{V}_{\psi} f(t_{0}, \omega)\right|=\left|\mathcal{V}_{\psi} g(t_{0}, \omega)\right|.
\end{align*}

(ii) $\Rightarrow$ (i). Assume that (\ref{equ1}) holds.
By the hypotheses,
  $f(\cdot+t_{0}) \overline{\phi(\cdot)}\in L^1(\mathbb{R})$
and $\operatorname{supp}( f(\cdot+t_{0}) \overline{\phi(\cdot)} )\subset [-B, B]$.
We see from the definition of Fourier transform that
\begin{align}\label{equ16}
M_{t_{0}}\mathcal{V}_{\phi} f(t_{0},\omega)=\int_{-B}^{B}f(x+t_{0}) \overline{\phi(x)} e^{-2 \pi i x \omega} \rmd x.
\end{align}
Similarly,
\begin{align}\label{equ16a}
M_{t_{0}}\mathcal{V}_{\phi} g(t_{0},\omega)=\int_{-B}^{B}g(x+t_{0}) \overline{\phi(x)} e^{-2 \pi i x \omega} \rmd x.
\end{align}
Note that
\begin{align}\label{equ3.5}
&|M_{t_{0}}\mathcal{V}_{\phi} f(t_{0},\omega+b)-M_{t_{0}}\mathcal{V}_{\phi} f(t_{0},\omega)| \nonumber \\=&\left|\int_{-B}^{B}f(x+t_{0}) \overline{\phi(x)} e^{-2 \pi i x \omega}(e^{-2 \pi i x b}-1) \rmd x\right| \nonumber \\=&\left|\int_{-B}^{B}f(x+t_{0}) \overline{\psi(x)} e^{-2 \pi i x \omega} \rmd x\right| \nonumber \\=&\left|\mathcal{V}_{\psi} f(t_{0},\omega)\right|.
\end{align}
By Lemma~\ref{Lm:L2} and (\ref{equ1}),   there exists some $\lambda_0 \in \mathbb{T}$ such that
\begin{align*}
f( x+t_{0}) \overline{\phi( x)}=\lambda_0 g( x+t_{0}) \overline{\phi( x)},\quad  a.e.
\end{align*}
or
\begin{align*}
f( x+t_{0}) \overline{\phi( x)}=\lambda_0 \overline{g(-x+t_{0})} \phi(-x),\quad a.e.
\end{align*}

Since $\overline{\phi(-x)}=\phi(x)\ne 0$, a.e. on $[-B,B]$, we have
\[
f(x) = \lambda_0 g(x),\quad \mbox{ a.e.  on } \, [t_0-B, t_0+B]
\]
or
\[
  f(x) = \lambda_0 \overline{g(-x+2t_{0})},\quad \mbox{ a.e.  on } \, [t_0-B, t_0+B].
\]
\end{proof}

To study STFT phase retrieval on the whole line,
we also need the following lemma.

\begin{lemma}\label{Lm:L3}
Let $B>0$ be a constant.
Suppose that $f,g$ are locally integrable functions on $\mathbb R$
and for almost all $t\in\mathbb R$,
there exist $\lambda(t)\in\mathbb T$ such that
\begin{equation}\label{eq:eb1}
  f(x) = \lambda(t) \overline{g(-x+2t)},
  \quad \mbox{ a.e. } x\in [t-B, t+B].
\end{equation}
Then $f(x) = \mu g(x) = c_0 e^{iax/2}$, a.e.
for some constants $\mu\in\mathbb T$,
 $c_0\in\mathbb C$ and $a\in\mathbb R$.
\end{lemma}

\begin{proof}
By the hypotheses, there is some $E\subset \mathbb R$
such that $|\mathbb R\setminus E|=0$ and for every $t\in E$, (\ref{eq:eb1}) holds.
We prove the conclusion in three steps.

(i)\, $|f|$ is a constant almost everywhere.

By (\ref{eq:eb1}),
for any $t\in E$,
\[
 | f(x)| = |g(-x+2t)|,
  \quad \mbox{ a.e. } x\in [t-B, t+B].
\]
Fix some $\varepsilon\in (0,B/3)$ and $t_0\in E\cap (0,\varepsilon)$. We have
\[
 | g(x)| = |f(-x+2t_0)|,
  \quad \mbox{ a.e. } x\in [t_0-B,t_0+B].
\]
And for any $t_1\in E\cap (0,\varepsilon)$,
\[
  |f(x)| = |g(-x+2t_1)| = |f(x-2(t_1-t_0))|,
  \quad \mbox{ a.e. } x\in [-(B-3\varepsilon), B-3\varepsilon].
\]
That is, $|f|$ is locally periodic
with period $2|t_1-t_0|$ which can be arbitrarily small.
By Lemma~\ref{Lm:L1}, we have
for almost all $ x\in  [ -(B-3\varepsilon), B-3\varepsilon ]$,
\begin{equation}\label{eq:eb7}
  |f(x)|=|g(x)|=C_0,
\end{equation}
where $C_0$ is a constant. Letting $\varepsilon\to 0$
yields that (\ref{eq:eb7}) is true for almost all
$x\in [-B,B]$.

Let
\begin{align*}
a_0 &= \inf\{a':\, \mbox{(\ref{eq:eb7}) holds for almost all } x\in [a', B] \}, \\
b_0 &= \sup\{b':\, \mbox{(\ref{eq:eb7}) holds for almost all } x\in [-B,b'] \}.
\end{align*}
Then $a_0\le -B$,
$b_0\ge B$ and
$|f|=|g|=C_0$ almost everywhere on $(a_0,b_0)$.
We conclude that $a_0=-\infty$
and $b_0=\infty$.

Assume on the contrary that $b_0<\infty$.
Take some $t_1 \in  E\cap [b_0-B/2, b_0-B/4]$.
When $x\in [ b_0, b_0+B/4]\subset [t_1-B,t_1+B]$,
$-x+2t_1 \in [b_0-5B/4, b_0-B/2]
\subset [-B, b_0]$. Hence
\begin{align*}
  |f(x)|&= |g(-x+2t_1)| =C_0,
  \quad \mbox{ a.e. } x\in [b_0, b_0+\frac{B}{4}],\\
  |g(x)|&= |f(-x+2t_1)| =C_0,
  \quad \mbox{ a.e. } x\in [b_0, b_0+\frac{B}{4}].
\end{align*}
Hence (\ref{eq:eb7}) holds almost everywhere
on $[b_0,b_0+B/4]$, which contradicts the choice of
$b_0$. Thus $b_0=\infty$.
Similar arguments yield that $a_0=-\infty$.

It suffices to consider the case $C_0>0$.

(ii) There exist
constants $\mu\in\mathbb T$,
 $c_0\in\mathbb C$, $a\in\mathbb R$ and $\delta>0$
 such that
 for almost all $x\in [-\delta, \delta]$,
\begin{equation}\label{eq:eb5}
  f(x) = \mu g(x)=c_0 e^{iax/2}.
\end{equation}

Fix some
$\varepsilon_0<B/12$
and $t_0\in E\cap [0,\varepsilon_0]$.
Denote
\[
  E_0 = E\cap (E-t_0),
\]
where $E-t_0 = \{t-t_0:\, t\in E\}$. Then $|\mathbb R\setminus E_0|=0$.
We see from
 (\ref{eq:eb1}) that
\[
  f(x) = \lambda(t_0) \overline{g(-x+2t_0)},\quad
  \mbox{ a.e. } x\in [t_0-B , t_0+B ].
\]
Hence
\begin{equation}\label{eq:eb4}
  g(x) = \lambda(t_0) \overline{f(-x+2t_0)},\quad
  \mbox{ a.e. } x\in  [t_0-B , t_0+B ].
\end{equation}
It follows that when $t\in E\cap [-B/3, B/3]$,
\[
  f(x) = \lambda(t) \overline{g(-x+2t)}
     = \lambda(t) \overline{\lambda(t_0) }
      f(x-2(t-t_0)),
      \quad
  \mbox{ a.e. } x\in [-\frac{B}{3}+\varepsilon_0, \frac{B}{3}-\varepsilon_0].
\]
Now the assumption $\varepsilon_0<B/12$ yields that
when $t\in E_0\cap [-B/4, B/4]$,
\begin{equation}\label{eq:eb3}
  f(x) = \lambda(t+t_0) \overline{g(-x+2(t+t_0))}
     = \lambda(t+t_0) \overline{\lambda(t_0) }
      f(x-2t),
      \quad
  \mbox{ a.e. } x\in
  \Big[-\frac{B}{4}, \frac{B}{4}\Big].
\end{equation}
Denote $\Lambda(t) = \lambda(t+t_0)\overline{\lambda(t_0) }$.
We obtain that
when $t\in E_0\cap [-B/4, B/4]$,
\begin{equation}\label{eq:eb3a}
  f(x) = \Lambda(t)
      f(x-2t),
      \quad
  \mbox{ a.e. } x\in
  \Big[-\frac{B}{4}, \frac{B}{4}\Big].
\end{equation}

For any $x_1<x_2$, we see from Fubini's theorem that
\[
  \int_{x_1}^{x_2} \int_{x_1}^{x_2}
    1_{E_0}(s)1_{E_0}(t) 1_{E_0}(s+t) \rmd s\rmd t
    = (x_2-x_1)^2.
\]
Hence for almost all $(s,t)\in E_0\times E_0$,
$s+t\in E_0$.
It follows that for almost all
$(s,t)$ in
\[
  F:= (E_0\times E_0 )\bigcap \Big(\Big[-\frac{B}{16},\frac{B}{16}\Big] \times \Big[-\frac{B}{16},\frac{B}{16}\Big]\Big),
\]
we have
\begin{align*}
f(x) &= \Lambda(s+t) f(x-2s-2t) \\
     &= \Lambda(s) f(x-2s)
      = \Lambda(s)\Lambda(t) f(x-2s-2t),
      \quad
  \mbox{ a.e. } x\in
  \Big[-\frac{B}{8}, \frac{B}{8}\Big].
\end{align*}
Since
$|f(x)|=C_0>0$, a.e.,
there is some $x_0$ such that $f(x_0)\ne 0$
and the above  equalities hold.
It follows  that for almost all $(s,t)\in F$,
\begin{equation}\label{eq:eb2}
\Lambda(s+t) = \Lambda(s)\Lambda(t).
\end{equation}

Note that  $f(x-2t)$ is measurable with respect to $(x,t)$.
We see from (\ref{eq:eb3a}) that $\Lambda(t)$ is measurable on $[-B/4, B/4]$.
On the other hand, the fact that $|\Lambda(t)|=1$ for all $t\in E_0$ yields that
there exist  $-B/16<t_1< t_2<B/16$ such that
$\int_{t_1}^{t_2} \Lambda(t) \rmd t\ne 0$.
Integrating both sides of (\ref{eq:eb2})
with respect to $t$ yields that
\[
    \int_{s+t_1}^{s+t_2} \Lambda(t) \rmd t
    = \Lambda(s) \int_{t_1}^{t_2} \Lambda(t) \rmd t,
    \quad \mbox{ a.e. } s\in E_0\cap[-\frac{B}{16}, \frac{B}{16}].
\]
Hence $\Lambda(s)$ equals a continuous function
$\tilde\Lambda(s)$ almost everywhere on $[-B/16,B/16]$.

Let $E_1$ be the subset of $E_0\cap [-B/16,B/16]$ such that
$\Lambda(s)=\tilde\Lambda(s)$ for all $s\in E_1$.
Then $|[-B/16,B/16]\setminus E_1|=0$.
It follows that
\[
  \int_{-B/32}^{B/32} \int_{-B/32}^{B/32}
    1_{E_1}(s)1_{E_1}(t) 1_{E_1}(s+t) \rmd s\rmd t
    = \frac{B^2}{256}.
\]
Consequently, for almost all $(s,t)\in E_2\times E_2$,
where $E_2=E_1\cap [-B/32,B/32]$,
$s+t\in E_1$. Hence
for almost all $s,t\in E_2$,
\begin{equation}\label{eq:eb2a}
  \tilde\Lambda(s+t) = \tilde\Lambda(s)\tilde\Lambda(t),
\end{equation}
Since $\tilde\Lambda$ is continuous on $[-B/16,B/16]$,
(\ref{eq:eb2a}) holds for all $s,t\in [-B/32,B/32]$.

Setting $s=t=0$ yields that $\tilde\Lambda(0)=1$. Hence
there is some $0<\delta<B/32$ such that
$|\arg(\tilde\Lambda(s))|<\pi/2$ when $|s|<\delta$.
Consequently, $\tilde \Lambda$ has a representation
\[
  \tilde \Lambda(t) = e^{ih(t)},
\]
where $h(t)$ is a continuous real valued function
for which $|h(t)|<\pi/2$ when $|t|<\delta$.
We see from (\ref{eq:eb2a}) that
for all $s,t\in [-\delta/2, \delta/2]$,
\[
  h(s+t) = h(s) + h(t).
\]
Since $h$ is continuous, we have
\[
  h(t) = at, \quad |t|<\delta
\]
for some constant $a\in\mathbb R$. Consequently,
\[
  \tilde \Lambda(t) = e^{iat}, \quad |t|<\delta.
\]
By (\ref{eq:eb3a}), for almost all $t\in E_0\cap [-\delta,\delta]$, we have
\[
f(x) = e^{iat} f(x-2t),
 \mbox{ a.e. } x\in \Big[-\frac{B}{4}, \frac{B}{4}\Big].
\]
Applying Fubini's theorem yields that
\begin{align*}
\int_{-B/4}^{B/4}
   \int_{-\delta}^{\delta} |f(x) - e^{iat} f(x-2t)| \rmd t \, \rmd x
&=
   \int_{-\delta}^{\delta} \int_{-B/4}^{B/4}
|f(x) - e^{iat} f(x-2t)|  \rmd x\, \rmd t = 0.
\end{align*}
Hence for almost all $x\in [-B/4, B/4]$,
\[
  f(x) = e^{iat} f(x-2t),\quad
  \mbox{ a.e. } t\in [-\delta,\delta].
\]
Since $|f(x)|>0$, a.e., there is
some  $x_0$ in $(-\delta/2,\delta/2)$ such that  $f(x_0)\ne 0$ and
\[
  f(x_0) = e^{iat} f(x_0-2t),\quad
  \mbox{ a.e. } t\in [-\delta,\delta].
\]
Hence
\[
  f(x) = c_0 e^{iax/2},
  \quad   \mbox{ a.e. } x\in [x_0-2\delta,x_0+2\delta],
\]
where $c_0$ is a constant. By (\ref{eq:eb1}),
there is some $t'_0\in E\cap (-\delta/4,\delta/4)$
such that
\[
  g(x) = \lambda(t'_0)\overline{f(-x+2t'_0)}
  =
  \lambda(t'_0)
   \bar c_0 e^{-iat_0} e^{iax/2},
  \,  \mbox{ a.e. } x\in [-x_0-2\delta+2t'_0,-x_0+2\delta+2t'_0].
\]
Since $|x_0|<\delta/2$,
we have
\[
  f(x) = \mu g(x)= c_0 e^{iax/2},\quad \mbox{ a.e. } x\in [-\delta, \delta],
\]
where $\mu = \overline{\lambda(t'_0)}  c_0e^{iat_0}/\bar c_0$.
That is, (\ref{eq:eb5}) holds for almost all $
x\in [-\delta, \delta]$.

(iii) (\ref{eq:eb5}) holds for almost all $
x\in \mathbb R$.

Let
\begin{align*}
a_0 &= \inf\{a':\, \mbox{(\ref{eq:eb5}) holds for almost all } x\in [a', \delta] \}, \\
b_0 &= \sup\{b':\, \mbox{(\ref{eq:eb5}) holds for almost all } x\in [-\delta,b'] \}.
\end{align*}
Then $a_0\le -\delta$,
$b_0\ge \delta$ and
(\ref{eq:eb5}) holds for almost all
 $ x\in [a_0,b_0]$.
We conclude that $a_0=-\infty$ and $b_0=\infty$.

Assume that $b_0<\infty$.
Take some $t_1\in E \cap [b_0-\delta/2,b_0-\delta/4]$. We have
\begin{equation}\label{eq:eb6}
  f(x) = \lambda(t_1)\overline{g(-x+2t_1)},
  \quad \mbox{ a.e. } x\in [t_1-B, t_1+B].
\end{equation}
When $x\in [b_0-\delta/2, b_0+\delta/2]
\subset [t_1-B,t_1+B]$,
we have $-x+2t_1\in [b_0-3\delta/2,b_0 ]
\subset [a_0,b_0]$.
Since (\ref{eq:eb5}) holds for almost all
 $ x\in [a_0,b_0]$, we obtain from (\ref{eq:eb6})
 that
\begin{align*}
f(x)
 &= \lambda(t_1)  \overline{g(-x+2t_1)}
  = \lambda(t_1)\mu \bar c_0 e^{-iat_1} e^{iax/2},
   \quad \mbox{ a.e. } x\in [b_0-\delta/2, b_0+\delta/2].
\end{align*}
Recall that $f(x)=c_0 e^{iax/2}$ a.e. on $[b_0-\delta/2,b_0]$.
We have
$
  \lambda(t_1)\mu \bar c_0 e^{-iat_1} = c_0$.
Hence
\[
   f(x) =  c_0 e^{iax/2},
   \quad \mbox{ a.e. } x\in [b_0, b_0+\delta/2].
\]

On the other hand, we see from (\ref{eq:eb6}) that
\[
    g(x) = \lambda(t_1)\overline{f(-x+2t_1)},
  \quad \mbox{ a.e. } x\in [t_1-B, t_1+B].
\]
Hence
\begin{align*}
g(x)&= \lambda(t_1)\bar c_0e^{-iat_1} e^{iax/2}
=\bar \mu f(x),
\quad \mbox{ a.e. } x\in [b_0,b_0+\delta/2].
\end{align*}
Now we obtain that (\ref{eq:eb5}) holds for almost all $x\in [b_0,b_0+\delta/2]$,
which contradicts  the choice of $b_0$.
Similarly we obtain $a_0=-\infty$.
\end{proof}

Now we are ready to state the main result
for STFT phase retrieval from
semi-discrete samples.

\begin{theorem}\label{thm1}
Let the window functions $\phi,\psi$ and the sampling sequence $\{\omega_{n}\}_{n\ge 1} $ be defined as in Lemma \ref{lem1}.
Let $f\in L_{loc}^1(\mathbb{R})$ be a function which is $2B$-nonseparable. Then $f$ is determined up to a global phase by the measurements $$\{|\mathcal{V}_{\phi} f(t,\omega_{n})|,~|\mathcal{V}_{\psi} f(t,\omega_{n})|:t \in \mathbb{R},~n \ge 1\}.$$
\end{theorem}

\begin{proof}
Suppose that a function $g\in L_{loc}^1(\mathbb{R})$ satisfies
\begin{align*}
\left\{\begin{array}{l}
|\mathcal{V}_{\phi} f(t,\omega_{n})|=|\mathcal{V}_{\phi} g(t,\omega_{n})|, \\
|\mathcal{V}_{\psi} f(t,\omega_{n})|=|\mathcal{V}_{\psi} g(t,\omega_{n})|,
\end{array}\right.
 \quad  \mbox{a.e. } t \in \mathbb{R},\,\,  \forall n \ge 1.
\end{align*}
By Lemma \ref{lem1}, there is some $E\subset \mathbb R$
such that $|\mathbb R\setminus E|=0$ and
for any $t \in E$, there exists some $\lambda(t)\in \mathbb{T}$ such that
\begin{align}\label{equ5}
f(x)=\lambda(t)g(x),\quad \mbox{ a.e.    } x\in [t-B,t+B]
\end{align}
or
\begin{align}\label{equ5-1}
f(x)=\lambda(t)\overline{g(-x+2t)},\quad \mbox{ a.e.  }x\in [t-B,t+B].
\end{align}

There are two cases.

(i)\,
 (\ref{equ5}) holds for some $t=t_0$.

In this case,
\begin{equation}\label{eq:ea9}
f(x)=\lambda(t_0)g(x),\quad \mbox{ a.e.    }
x\in [ t_0-B, t_0+B ].
\end{equation}
Let
\begin{align*}
  a_0 &= \inf\{a'\le t_0-B:  f(x) = \lambda(t_0) g(x),
  \mbox{ a.e. on } [a', t_0+B] \},
\\
  b_0 &= \sup\{b'\ge t_0+B:  f(x) = \lambda(t_0) g(x),
  \mbox{ a.e. on } [t_0-B, b'] \}.
\end{align*}
Then $a_0\le t_0-B$,
$b_0\ge t_0+B$, and
\begin{equation}\label{eq:eb11}
  f(x) = \lambda(t_0) g(x),
  \mbox{ a.e. on } [a_0, b_0] .
\end{equation}
We conclude that $a_0=-\infty$
and $b_0=\infty$.

Assume on the contrary that  $b_0<\infty$. Since $f$ is $2B$-nonseparable, we have
\[
  \int_{b_0-2B}^{b_0} |f(x)| \rmd x >0.
\]
By the monotone  convergence theorem, we obtain
that
\[
  \lim_{\varepsilon\to 0}
  \int_{b_0-2B+\varepsilon}^{b_0} |f(x)| \rmd x
  =\int_{b_0-2B}^{b_0} |f(x)| \rmd x >0.
\]
Hence there is some $0<\varepsilon_0<B $ such that
\begin{equation}\label{eq:ea7}
  \int_{b_0-2B+\varepsilon_0}^{b_0} |f(x)| \rmd x >0.
\end{equation}
We conclude that
 on $[b_0,b_0+\varepsilon_0/2]$,
\begin{equation}\label{eq:ea8}
f(x)=\lambda(t_0)g(x),\quad a.e.
\end{equation}
To see this, take some $t_1 \in E\cap [b_0-B+3\varepsilon_0/4, b_0-B+\varepsilon_0] $.
If (\ref{equ5}) holds for $t=t_1$, then
\begin{equation}\label{eq:eb12}
  f(x) = \lambda(t_1)
     g(x),\quad  \mbox{ a.e. on } [t_1-B,t_1+B].
\end{equation}
Now it follows from (\ref{eq:eb11}) that
\[
  f(x) = \lambda(t_0) g(x) = \lambda(t_1)
     g(x),\quad  \mbox{ a.e. on } [b_0-2B+\varepsilon_0, b_0].
\]
Since $f\ne 0$ on a subset of $[b_0-2B+\varepsilon_0, b_0]$ with positive measure,
thanks to (\ref{eq:ea7}),
we have $\lambda(t_0) = \lambda(t_1)$.
Now we see from (\ref{eq:eb12}) that
(\ref{eq:ea8}) holds on $[b_0,b_0+\varepsilon_0/2]$.

If (\ref{equ5-1}) holds for $t=t_1$,
then
\[
  f(x)=\lambda(t_0)g(x)
  = \lambda(t_1) \overline{g(-x+2t_1)},\quad  \mbox{ a.e. on } [b_0-2B+\varepsilon_0, b_0].
\]
Hence
\begin{equation}\label{eq:eb13}
  g(x)
  = \overline{\lambda(t_0)}\lambda(t_1) \overline{g(-x+2t_1)},\quad  \mbox{ a.e. on } [b_0-2B+\varepsilon_0, b_0].
\end{equation}
Note that for $y\in [b_0,b_0+\varepsilon_0/2]$,
$-y+2t_1 \in [b_0-2B+\varepsilon_0, b_0]$.
By (\ref{equ5-1}) and (\ref{eq:eb13}),
for almost all $y\in [b_0,b_0+\varepsilon_0/2]$,
\begin{align*}
f(y)
  &= \lambda(t_1) \overline{g(-y+2t_1)}
   = \lambda(t_1) \lambda(t_0) \overline{\lambda(t_1)}
      g(y)
   =    \lambda(t_0)      g(y).
\end{align*}
Hence
(\ref{eq:ea8}) also holds on $[b_0,b_0+\varepsilon_0/2]$.
In both cases, (\ref{eq:ea8}) holds on $[b_0,b_0+\varepsilon_0/2]$, which contradicts the
choice of $b_0$.

Similar arguments yield  that $a_0=-\infty$. Hence
\[
  f(x) = \lambda(t_0) g(x), \quad a.e.
\]

(ii) (\ref{equ5-1}) holds for all $t\in E$.

In this case, we get the conclusion by Lemma~\ref{Lm:L3}.
\end{proof}

\begin{remark}
The hypothesis that $f$ is $2B$-nonseparable
is necessary in Theorem~\ref{thm1}.
\end{remark}

For example, suppose that
 $f(x) = 0$ on $[t_0-B, t_0+B]$
for some $t_0$.
Let
\[
  g = f\cdot  1_{(-\infty, t_0-B)}
     + \lambda f\cdot  1_{(t_0+B,\infty)},
\]
where $\lambda\in\mathbb T\setminus\{1\}$ is a constant.
Then we have
\[
  |\mathcal V_{\phi}f(t,\omega)|
 = |\mathcal V_{\phi}g(t,\omega)|,
 \quad \mbox{ a.e. } (t,\omega)\in \mathbb R^2,
\]
while $f\not\sim g$.

\section{Phase retrieval from discrete samples}

In this section, we study STFT phase retrieval from
discrete samples.
We consider three types of
signals which include periodic,
quasi-periodic, and locally integrable signals.
For all types of signals we show
that STFT phase retrieval holds for nonseparable signals.

\subsection{Periodic signals}
First, we consider period signals.
Given a positive number $T$, denote by
$L^1(T)$ the function space
consisting of all locally integrable functions
which are periodic with a period $T$.

\begin{theorem}\label{thm3a}
Let the window functions $\phi,\psi$ and the sampling sequence $\{\omega_{n}\}_{n\ge 1}\subset\mathbb{R}$ be defined as in Lemma~\ref{lem1}.
Suppose that $0<T\le 2B$ and $f\in L^1(T)$.
Let $t_0, t_1\in  \mathbb{R}$ be such that
both $f\phi(\cdot-t_0)$ and
$f\phi(\cdot-t_1)$  are integrable.
If $(t_1-t_0)/T$ is an irrational number,
then $f$ is determined up to a global phase by the measurements $$\{|\mathcal{V}_{\phi} f(t_{j},\omega_{n})|,|\mathcal{V}_{\psi} f(t_{j},\omega_{n})|:n \ge 1,~j={0,1}\}.$$

Moreover, the conclusion fails if
$(t_1-t_0)/T$ is a rational number.
\end{theorem}

\begin{proof}
First, we assume
that  $(t_1-t_0)/T$ is an irrational number.
Suppose that two functions $f, g\in L^1(T)$ satisfy
\begin{align}
\left\{\begin{array}{l}
|\mathcal{V}_{\phi} f(t_{j},\omega_{n})|=|\mathcal{V}_{\phi} g(t_{j},\omega_{n})|, \\
|\mathcal{V}_{\psi} f(t_{j},\omega_{n})|=|\mathcal{V}_{\psi} g(t_{j},\omega_{n})|,
\end{array}\right.
 \quad \forall n \ge 1,~j={0,1}.
 \label{eq:eb9}
\end{align}
By Lemma \ref{lem1}, for $j={0,1}$,
there exists some $\lambda_j  \in \mathbb{T}$ such that
\[
  f(x) = \lambda_j g(x),\quad \mbox{ a.e. } x\in [t_j-B, t_j+B]
\]
or
\[
  f(x) = \lambda_j \overline{g(-x+2t_j)},\quad \mbox{ a.e. } x\in [t_j-B, t_j+B].
\]
Since both functions  have the period $T$,
we obtain that
\begin{align*}
f=\lambda_j  g~\text{or} ~f=\lambda_j  \overline{g(2t_j-\cdot)}.
\end{align*}
It suffices to consider the case
\[
 f=\lambda_0  \overline{g(2t_0-\cdot)}
 = \lambda_1  \overline{g(2t_1-\cdot)}.
\]
Denote $a= t_1-t_0$. We have
\begin{align*}
f(x)&=\lambda_0 \overline{g(2t_{0}-x)}=\lambda_0 \overline{\lambda_1 }
f(x+2t_1-2t_0)\nonumber\\
&=\lambda_0 \bar\lambda_1
f(x+2a),\quad a.e.
\end{align*}

Recall that $f$ has the period $T$. The $k$-th Fourier coefficient of $f$ is
\begin{align*}
\hat f(k) &= \int_0^T f(x) e^{-2\pi i kx/T} \rmd x \\
&= \lambda_0 \bar\lambda_1 \int_0^T f(x+2a) e^{-2\pi i kx/T} \rmd x \\
&= \lambda_0 \bar\lambda_1
e^{4\pi i ka/T}
\int_{2a}^{2a+T} f(x) e^{-2\pi i kx/T} \rmd x \\
&= \lambda_0 \bar\lambda_1
e^{4\pi i ka/T} \hat f(k),\quad \forall k\in\mathbb Z.
\end{align*}
Hence
\begin{equation}\label{eq:eb8}
    \hat f(k)
    (\lambda_0 \bar\lambda_1
e^{4\pi i ka/T}-1) = 0,\quad \forall k\in\mathbb Z.
\end{equation}
If there are two integers $k, k'$ satisfying
\[
  e^{4\pi i ka/T} = e^{4\pi i k'a/T},
\]
then $4\pi  (k-k')a/T = 2l\pi$ for some integer $l$.
Hence  $a/T$ is a  rational number,
which contradicts the hypothesis.
Consequently,
there is
only one integer $k_0$ such that
\[
  \lambda_0 \bar\lambda_1
e^{4\pi i k_0a/T}-1=0.
\]
Therefore, $\hat f(k)=0$ for all $k\ne k_0$. It follows that
\[
  f(x) = C e^{ 2\pi i k_0x/T},\quad a.e.
\]
Hence
\[
  g(x) = \lambda_0 \overline{f(2t_0-x)}
    = \lambda_0 \bar C e^{ -4\pi i k_0t_0/T}
    e^{ 2\pi i k_0x/T},\quad a.e.
\]
Consequently,
we also have $f = \mu g$ for some $\mu\in\mathbb T$.

Next we assume that
$a/T$ is a rational number.
Suppose that $2a/T=p/q$, where $p$ and $q$ are coprime positive integers.
We see from (\ref{eq:eb8})
that $\hat f(k)$ might be nonzero whenever
$\lambda_0\bar \lambda_1 e^{2\pi i k p/q}=1$.
As a result, it is possible that $f/g$ is not a constant.
For an explicit example, see Example~\ref{ex:ex2}.
This completes the proof.
\end{proof}

\begin{example}\label{ex:ex2}
If $(t_1-t_0)/T$ is a rational number,
then there exist functions
$f,g\in L^1(T)$ which meet (\ref{eq:eb9})
while $f\not\sim g$.
\end{example}

\begin{proof}
Suppose that $2(t_1-t_0) = pT/q$, where $p, q$
are integers, $q>0$.
Let
\begin{align*}
f(x) &= c_0+c_q e^{2\pi i qx/T},
\\
g(x) &= \bar c_0+\bar c_q e^{-4\pi i q t_0/T} e^{2\pi i qx/T},
\end{align*}
where $c_0, c_q\ne 0$.
We have
\[
g(x)  = \bar c_0+\bar c_q e^{-4\pi i q t_1/T} e^{2\pi i qx/T} .
\]
Hence
\[
  f(x) = \overline{g(2t_0-x)}
  =\overline{g(2t_1-x)}.
\]
Therefore, $f$ and $g$ meet (\ref{eq:eb9}).
However, if
\[
  \frac{\bar c_0 }{c_0}
  \ne
  \frac{\bar c_q }{c_q}e^{-4\pi i q t_0/T},
\]
then there is no constant $\lambda$ satisfying
\[
  f=\lambda g.
\]
\end{proof}

\subsection{Quasi-periodic signals}
Next we consider quasi-periodic signals.
Given a positive number $T$,
we call a function $f$
quasi-periodic with a period $T$
if  there is some $\mu\in\mathbb T$ such that
\[
  f(x)=\mu f(x+T), \mbox{ a.e. on } \mathbb{R}.
\]
Denote by $\tilde L_{\mu}(T)$
the space consisting of all functions
satisfying the above equation.
Let
\begin{align*}
\tilde L (T)&= \bigcup_{\mu\in\mathbb T}\tilde L_{\mu}(T).
\end{align*}

Suppose that $\mu=e^{i\theta}$. For any $f\in \tilde L_{\mu}(T)$,
$e^{i\theta x/T} f(x)$ is a periodic function with a period $T$.

\begin{lemma}\label{Lm:La0}
Suppose $0<T<2B$.
Let $f,g\in \tilde L(T)$ be such that
either
\begin{equation}\label{eq:eb15}
  f(x)=\lambda_0  g(x),\quad \mbox{ a.e. } x\in[t_{0}-B,t_{0}+B]
\end{equation}
 or
\begin{equation}\label{eq:eb16}
 f(x)=\lambda_0  \overline{g(-x+2t_{0})},\quad \mbox{ a.e. } x\in[t_{0}-B,t_{0}+B].
\end{equation}
Then $f, g\in \tilde L_{\mu}(T)$ for some $\mu\in \mathbb{T}$
and
 (\ref{eq:eb15}) or (\ref{eq:eb16}) holds almost everywhere on $\mathbb{R}$
 provided   one of the following conditions holds,

\begin{enumerate}
\item \label{La0:a} $0<T\le B$;

\item \label{La0:b} $B<T<2B$ and neither $f$ nor $g$ is $(2B-T)$-separable.
\end{enumerate}
\end{lemma}

\begin{proof}
Assume that $f\in \tilde L_{\mu}$ for some  $\mu \in \mathbb{T}$.

First we suppose that (\ref{eq:eb15}) holds.
Then for Case (\ref{La0:a}), we have
$$g(x)=\overline{\lambda_0 } f(x)=\overline{\lambda_0} \mu f(x+T)=\mu g(x+T), \quad \mbox{ a.e. } x \in[t_{0}-B,t_{0}-B+T].$$
Hence $g\in \tilde L_{\mu}(T)$.

For Case (\ref{La0:b}), we have
$$g(x)=\overline{\lambda_0 } f(x)=\overline{\lambda_0} \mu f(x+T)=\mu g(x+T), \quad \mbox{ a.e. } x \in[t_{0}-B,t_{0}+B-T].$$
Since $g$ is $(2B-T)$-nonseparable, we have $g\in \tilde L_{\mu}(T)$.

It follows from (\ref{eq:eb15}) that
for any $l\in\mathbb Z$
and
  almost all $x\in [t_{0}-B+lT,t_{0}-B+(l+1)T]$,
\begin{align*}
f(x)&=\overline{\mu}^{l} f(x-lT)
 =\lambda_0 \overline{\mu}^{l}g(x-lT)
=\lambda_0 g(x).
\end{align*}
Hence
\[
f(x)=\lambda_0 g(x),
\qquad \mbox{ a.e.  on } \mathbb R.
\]

Next we suppose that (\ref{eq:eb16}) holds.
For Case (\ref{La0:a}), we have
\begin{align*}
g(x) & =\lambda_0  \overline{f(-x+2 t_{0})}=\lambda_0 \mu \overline{f(-x-T+2 t_{0})}
  =\mu g(x+T),   \mbox{ a.e. }~ x \in [t_{0}-B,t_{0}-B+T].
\end{align*}
Hence $g\in \tilde L_{\mu}(T)$.

For Case (\ref{La0:b}), we have
\begin{align*}
g(x) & =\lambda_0  \overline{f(-x+2 t_{0})}=\lambda_0 \mu \overline{f(-x-T+2 t_{0})}
  =\mu g(x+T),   \mbox{ a.e. }~ x \in [t_{0}-B,t_{0}+B-T].
\end{align*}
Again, we obtain $g\in \tilde L_{\mu}(T)$.
It follows from (\ref{eq:eb16}) that
for any $l\in\mathbb Z$
and
  almost all $x\in [t_{0}-B+lT,t_{0}-B+(l+1)T]$,
\begin{align*}
f(x)&=\overline{\mu}^{l} f(x-lT)
 =\lambda_0 \overline{\mu^{l}g(-x+lT+2t_0)}
=\lambda_0  \overline{g(-x +2t_0)}.
\end{align*}
Hence $f(x) =\lambda_0  \overline{g(-x +2t_0)}$ a.e. on $\mathbb R$.
\end{proof}

With Lemma~\ref{Lm:La0}, we   get a uniqueness result
for STFT phase retrieval of quasi-periodic signals.

\begin{theorem}\label{thm:th3}
Let the window functions $\phi,\psi$ and the sampling sequence $\{\omega_{n}\}_{n\ge 1} $ be defined as in Lemma \ref{lem1} and $0<T < 2B$.
Suppose that
$f \in \tilde L (T)$
and $t_0$, $t_1$ are real numbers such that
both $f\phi(\cdot-t_0)$ and
$f\phi(\cdot-t_1)$  are integrable.
If $(t_1-t_0)/T$ is an irrational number
and $f$ is $(2B-T)$-nonseparable
when $B<T<2B$,
then $f$ is determined up to a global phase by the measurements \begin{align}\label{equ1.3-1}
\{|\mathcal{V}_{\phi} f(t_{j},\omega_{n})|,|\mathcal{V}_{\psi} f(t_{j},\omega_{n})|:n \ge 1,~j={0,1}\}.
\end{align}

Moreover, the conclusion fails if
$(t_1-t_0)/T$ is a rational number
or
$f$ is
$(2B-T +\varepsilon)$-separable
when $B<T<2B$, where $\varepsilon>0$.
\end{theorem}

\begin{proof}
Since $f\in \tilde L(T)$, there is some
$\mu \in \mathbb{T}$ such that
\begin{align}\label{equ1.3}
f(x) = \mu f(x + T),\quad \mbox{ a.e. }
\end{align}
Suppose that a function $g\in \tilde L (T)$ satisfies
\begin{align}
\left\{\begin{array}{l}
|\mathcal{V}_{\phi} f(t_{j},\omega_{n})|=|\mathcal{V}_{\phi} g(t_{j},\omega_{n})|, \\
|\mathcal{V}_{\psi} f(t_{j},\omega_{n})|=|\mathcal{V}_{\psi} g(t_{j},\omega_{n})|,
\end{array}\right.
\quad \forall n\ge 1, j= 0,1 .
\label{eq:ed1}
\end{align}
By Lemmas~\ref{lem1} and \ref{Lm:La0},
$g\in\tilde L_{\mu}(T)$.
Moreover,
for $j=0,1$,
there exists some $\lambda_j \in \mathbb{T}$ such that
\begin{align*}
f=\lambda_j g,  \quad \mbox{ a.e. on } \mathbb R
\end{align*}
or
\begin{align*}
f=\lambda_j  \overline{g(2t_j-\cdot)},  \quad \mbox{ a.e. on } \mathbb R.
\end{align*}

It suffices to consider the case
\[
 f=\lambda_0  \overline{g(2t_0-\cdot)}
 = \lambda_1  \overline{g(2t_1-\cdot)},
 \quad \mbox{ a.e. on } \mathbb R.
\]
Denote $a= t_1-t_0$. We have
\begin{align}\label{equ1.4}
f(x)&=\lambda_0\overline{g(2t_{0}-x)}=\lambda_0\overline{\lambda_1}
f(x+2t_1-2t_{0})\nonumber\\&=\lambda_0\overline{\lambda_1}
f(x+2a),\quad \mbox{ a.e. }x\in\mathbb{R}.
\end{align}
Let $\mu=e^{i\theta_{1}}$, $\lambda_0\overline{\lambda_1}=e^{i\theta_{2}}$ and $h(x)=e^{i {\theta_{1}}x/{T}}f(x)$. From (\ref{equ1.3}) and (\ref{equ1.4}), we deduce that
\[
h(x+2a)=e^{2ia\theta_{1}/T}  \overline{\lambda_0}\lambda_1h(x)
:= e^{i\theta} h(x) ,\quad \mbox{ a.e. }
\]
and
\[
h(x+T)=h(x),\quad \mbox{ a.e. }
\]
That is,  $h$ has the period $T$.
Similarly to (\ref{eq:eb8}) we obtain
that the $k$-th Fourier coefficient of $h$
satisfies
\begin{align}\label{eq:eb14}
    \hat h(k)
    (e^{-i\theta}
e^{4\pi i ka/T}-1) = 0,\quad \forall k\in\mathbb Z.
\end{align}
If there are two integers $k, k'$ such that
\[
  e^{4\pi i ka/T} = e^{4\pi i k'a/T},
\]
then $4\pi  (k-k')a/T = 2l\pi$ for some integer $l$.
Hence  $a/T$ is a  rational number,
which contradicts the hypothesis.
Consequently,
there is
only one integer $k_0$ such that
\[
  e^{-i\theta}
e^{4\pi i k_0a/T}-1=0.
\]
Therefore, $\hat h(k)=0$ for all $k\ne k_0$. It follows that
\[
  h(x) = C e^{ 2\pi i k_0x/T},\quad a.e.
\]
Hence
\[
  f(x) =Ce^{ i(2\pi k_{0}-\theta_{1})x/T},\quad \mbox{ a.e. }
\]
And
\[
  g(x) = \lambda_0 \overline{f(2t_0-x)}
    = \lambda_0 \bar C e^{ -2i(2\pi k_{0}-\theta_{1})t_0/T}
    e^{ i(2\pi k_{0}-\theta_{1})x/T},\quad a.e.
\]
Consequently,
we also have $f = \lambda g$ for some $\lambda\in\mathbb T$.

On the other hand, note
that a period function is also quasi-periodic.
Example~\ref{ex:ex2} also works for this case.
That is,
if
$(t_1-t_0)/T$ is a rational number,
then there exist functions $f,g\in\tilde L(T)$
such that  (\ref{eq:ed1}) holds
while $f\not\sim g$.

For a counterexample
when $f$ is $(2B-T +\varepsilon)$-separable  with $B<T<2B$,
see Example~\ref{ex:ex1a} below.
This completes the proof.
\end{proof}

\begin{example} \label{ex:ex1a}
For $B<T<2B$, if $f$ is $(2B-T)$-separable, then the measurements
(\ref{equ1.3-1}) fail to determine $f$ uniquely up to a global phase.
\end{example}

\begin{proof}
Take some
$\alpha \in (2B/T-1,1)\backslash \mathbb Q$. Set $t_0=0$, $t_1=\alpha T$,
\[
f(x) =\begin{cases}
0, & x\in(-B,B-T)\cup (-B+\alpha T,-B+T),\\
c,
& x\in(B-T,-B+\alpha T),
\end{cases}
\]
where $c\ne0$ is a constant,
$f(x)=f(x+T)$,
$g(x)=f(x)$ on $(-B,-B+T)$ and
$g(x)=-g(x+T)$. See the figure below.

%
%
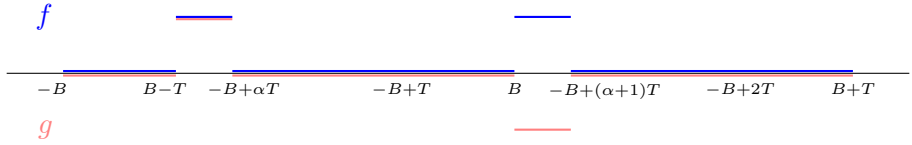
\begin{figure}[!ht]
\begin{tikzpicture}[scale=1.5]
\draw[blue,thick] (-2,0.02) -- (-1,0.02);
\draw[blue,thick] (-1,0.5) -- (-0.5,0.5);
\draw[blue,thick] (-0.5,0.02) -- (2,0.02);
\draw[blue,thick] (2,0.5) -- (2.5,0.5);
\draw[blue,thick] (2.5,0.02) -- (5,0.02);
\draw (-2.5,0)--(5.5,0);
\draw[red!50,thick] (-2,-0.02) -- (-1,-0.02);
\draw[red!50,thick] (-1,0.48) -- (-0.5,0.48);
\draw[red!50,thick] (-0.5,-0.02) -- (2,-0.02);
\draw[red!50,thick] (2,-0.5) -- (2.5,-0.5);
\draw[red!50,thick] (2.5,-0.02) -- (5,-0.02);
\coordinate [label=below:$\scriptscriptstyle -B$](A) at (-2.1,0);
\coordinate [label=below:$\scriptscriptstyle B-T$](B) at (-1.1,0);
\coordinate [label=below:$\scriptscriptstyle -B+\alpha T$]
  (C) at (-0.4,0);
\coordinate [label=below:$\scriptscriptstyle -B+T$](D) at (1,0);
\coordinate [label=below:$\scriptscriptstyle B$](E) at (2,0);
\coordinate [label=below:$\scriptscriptstyle -B+(\alpha+1)T$](F) at (2.8,0);

\coordinate [label=below:$\scriptscriptstyle -B+2T$](G) at (4,0);

\coordinate [label=below:$\scriptscriptstyle B+T$](H) at (5,0);

\coordinate [label=left:${\color{blue} f}$](I) at (-2,0.5);
\coordinate [ label=left:${\color{red!50} g}$](J) at (-2,-0.5);
\end{tikzpicture}
\caption{Two signals $f$ and $g$. }
\label{fig:f1}
\end{figure}

Then we have
$f=g=0$ on $(-B+T,B)\cup (-B+(\alpha+1)T,B+T)$.
Hence
\[
\begin{cases}
f(x) =g(x), & x\in(-B,B),\\
f(x) =-g(x),
& x\in(\alpha T-B,\alpha T+B).
\end{cases}
\]
Therefore, $f\not\sim g$.
However,
\begin{align*}
|\mathcal{V}_{\phi} f(t_j,\omega)|&=\left|\int_{-B}^{B}f(x+t_j) \overline{\phi(x)} e^{-2 \pi i x \omega} \rmd x\right|\\
&=\left|\int_{-B}^{B}g(x+ t_j) \overline{\phi(x)} e^{-2 \pi i x \omega} \rmd x\right|\\
&=|\mathcal{V}_{\phi} g(t_j,\omega)|,
\quad \forall \omega\in\mathbb R
\end{align*}
for $j=0,1$.
\end{proof}

\subsection{Locally integrable signals}

In this subsection, we study STFT phase retrieval
for locally integrable signals from discrete samples.

Theorem~\ref{thm:main} states that
when $\phi\in L^{p'}$ and $f\in L^p$,
(\ref{eq:ed2}) determines
$f$ up to a constant factor.
However, the conclusion is false if we
extend the signal class
from $L^p$ to $L_{loc}^p$.
In this case, one more
sampling point for the parameter $t$
is needed.

\begin{theorem}\label{thm1a}
Let $1\le p\le \infty$ and $\phi\in L^{p'}(\mathbb R)$
satisfy (\ref{eq:phi}).
Let the window function $\psi$ and the sampling sequence $\{\omega_{n}\}_{n\ge 1} $ be defined as in Lemma \ref{lem1}. Let $0<a\le B$ and $f\in L^{p}_{loc}(\mathbb R)$
which is  $(2B-a)$-nonseparable. Then $f$ is determined up to a global phase by the measurements
\begin{align}\label{equ1.6}
\{|\mathcal{V}_{\phi} f(ma,\omega_{n})|,~|\mathcal{V}_{\psi} f(ma,\omega_{n})|,~|\mathcal{V}_{\phi} f(t_0,\omega_{n})|,~|\mathcal{V}_{\psi} f(t_0,\omega_{n})|:n\ge 1, m \in \mathbb{Z}\},
\end{align}
where $t_0/a$ is an irrational number.

Moreover,
the conclusion fails if one of the following conditions holds,
\begin{enumerate}
\item the sampling interval $a$ is greater than $B$;
\item the signal $f$ is $(2B-a)$-separable;
\item the ratio $t_0/a$ is a rational number;
\item the measurements
$
\{|\mathcal{V}_{\phi} f(t_0,\omega_{n})|$,
$
|\mathcal{V}_{\psi} f(t_0,\omega_{n})|:\,
n \ge 1\}$
are removed.
\end{enumerate}
\end{theorem}

\begin{proof}
Suppose that a function $g\in L^{p}_{loc}(\mathbb R)$ satisfies
\begin{align}
\left\{\begin{array}{l}
|\mathcal{V}_{\phi} f(ma,\omega_{n})|=|\mathcal{V}_{\phi} g(ma,\omega_{n})|, \\
|\mathcal{V}_{\psi} f(ma,\omega_{n})|=|\mathcal{V}_{\psi} g(ma,\omega_{n})|,\\
|\mathcal{V}_{\phi} f(t_0,\omega_{n})|=|\mathcal{V}_{\phi} g(t_0,\omega_{n})|, \\
|\mathcal{V}_{\psi} f(t_0,\omega_{n})|=|\mathcal{V}_{\psi} g(t_0,\omega_{n})|,
\end{array}\right.
 \quad \forall n\ge 1,\, m \in \mathbb{Z}.
 \label{eq:ec3}
\end{align}
By Lemma \ref{lem1}, for any $m \in \mathbb{Z}$,
there exists some $\lambda_m \in \mathbb{T}$ such that
\begin{align}\label{equ1.5}
  f(x) = \lambda_m g(x),\quad \mbox{ a.e. } x\in [ma-B, ma+B]
\end{align}
or
\begin{align}\label{equ1.5-1}
  f(x) = \lambda_m \overline{g(-x+2ma)},\quad \mbox{ a.e. } x\in [ma-B, ma+B].
\end{align}

There are two cases.

(i)\,
 (\ref{equ1.5}) holds for some $m=m_0$. Without loss of generality, we may assume $m_0=0$.

In this case,
\begin{equation}\label{eq:ec1}
f(x)=\lambda_{0}g(x),\quad \mbox{ a.e. } x\in [-B,B].
\end{equation}
We prove by induction that for any $m\ge 0$.
\begin{equation}\label{eq:ec2}
  f(x)=\lambda_{0}g(x),\quad \mbox{ a.e. } x\in [-B-ma,B+ma].
\end{equation}

First, (\ref{eq:ec2}) holds for $m=0$, thanks to
(\ref{eq:ec1}).

Assume that (\ref{eq:ec2}) holds for some $m=k\ge 0$.
For $m=k+1$,  if (\ref{equ1.5}) holds, that is,
$f(x)=\lambda_{k+1}g(x)$, a.e. $x\in[(k+1)a-B,(k+1)a+B]$,
then
\begin{align*}
f(x)=\lambda_{0}g(x)=\lambda_{k+1}g(x),\quad \mbox{ a.e. } x\in [(k+1)a-B, ka+B].
\end{align*}
Since $f$ is $(2B-a)$-nonseparable, we deduce that $\lambda_{0}=\lambda_{k+1}$.
Hence
\begin{equation}\label{eq:ed6}
    f(x)=\lambda_{0}g(x),\quad \mbox{ a.e. } x\in [ka+B, (k+1)a+B].
\end{equation}

If (\ref{equ1.5-1}) holds, then
\begin{equation}\label{eq:ec4}
  f(x)=\lambda_{k+1}\overline{g(-x+2(k+1)a)},
  \quad \mbox{ a.e. } x\in [(k+1)a-B, (k+1)a+B],
\end{equation}
which is equivalent to
\begin{equation}\label{eq:ec4a}
  g(x)=\lambda_{k+1}\overline{f(-x+2(k+1)a)},
  \quad \mbox{ a.e. } x\in [(k+1)a-B, (k+1)a+B].
\end{equation}
Since $f(x) = \lambda_0 g(x)$ a.e. on
$[(k+1)a-B,ka+B]$, we have
\[
  \overline{g(-x+2(k+1)a)}
  =\lambda_0 \overline{f(-x+2(k+1)a)},
  \quad \mbox{ a.e. } x\in   [(k+1)a-B,ka+B].
\]
That is,
\[
  f(y)
  =\lambda_0 g(y),
  \quad \mbox{ a.e. } y\in   [(k+2)a-B, (k+1)a+B].
\]
Observe that $(k+2)a-B\le ka+B$.
Hence (\ref{eq:ed6}) also holds.
Similarly we get the identity
on $[-B-(k+1)a, -B-ka]$. Hence (\ref{eq:ec2}) holds for $m=k+1$. By induction, it holds for all $m\ge 0$.
That is, $f = \lambda_0 g$, a.e.

(ii) (\ref{equ1.5-1}) holds for all $m \in \mathbb{Z}$.

Then it is easy to show that
\begin{align}
f(x)&=\lambda_{m-1}\overline{g(-x+2(m-1)a)}
 =\lambda_{m-1}\overline{\lambda_{m}}f(x  +2 a),
   \nonumber \\
 &
   \qquad \mbox{ a.e. } x\in [(m-1)a-B, (m-2)a+B],~m\in\mathbb{Z}.
  \label{eq:ec5}
\end{align}
Since $\cup_{m\in\mathbb{Z}}[(m-1)a-B, (m-2)a+B]=\mathbb{R}$, we have
\begin{align}\label{equ1.19}
|f(x)|= |f(x+2a)|,\quad \mbox{ a.e. }
 x\in\mathbb{R}.
\end{align}
Similarly,
\begin{align}
g(x)
 =\lambda_{m-1}\overline{\lambda_{m}}g(x  +2 a),
 \quad \mbox{ a.e. } x\in [(m-1)a-B, (m-2)a+B],~m\in\mathbb{Z}.
    \label{eq:ec5-1}
\end{align}
and
\begin{align*}
|g(x)|= |g(x+2a)|,\quad \mbox{ a.e. } x\in\mathbb{R}.
\end{align*}

On the other hand, applying Lemma \ref{lem1} to (\ref{eq:ec3}) yields that there exists some $\lambda(t_0) \in \mathbb{T}$ such that
\begin{align}\label{equ1.15}
  f(x) = \lambda(t_0) g(x),\quad \mbox{ a.e. } x\in [t_0-B, t_0+B]
\end{align}
or
\begin{align}\label{equ1.15-1}
  f(x) = \lambda(t_0) \overline{g(-x+2t_0)},\quad \mbox{ a.e. } x\in [t_0-B, t_0+B].
\end{align}
Since $t_{0}/a$ is an irrational number, there exists a unique integer $m_0$ such that
\begin{align*}
(m_0-1)a < t_{0} < m_0a.
\end{align*}

If (\ref{equ1.15}) holds, then
setting $m=m_0$ in (\ref{equ1.5-1}) yields that
\begin{align*}
g(x)=\lambda_{m_0}\overline{f(-x+2m_0a)}
=\overline{\lambda(t_{0})}
\lambda_{m_0}\overline{g(-x+2m_0a)},\quad \mbox{ a.e. } x\in [t_{0}+B,m_0a+B].
\end{align*}
Consequently,
\[
f(x)=\lambda_{m_0}\overline{g(-x+2m_0a)}=\lambda(t_{0})g(x),\quad \mbox{ a.e. } x\in [t_{0}+B,m_0a+B].
\]
Hence, we obtain that
$f(x)=\lambda(t_{0})g(x)$, a.e. on $[m_0a-B,m_0a+B]$. It follows  from Case~(i) that
$f = \lambda(t_0) g$, a.e.

If (\ref{equ1.15-1})  holds, then
for any $x\in [t_0-B, t_0]$, we have $ -x+2t_0\in[t_0, t_0+B] \subset[m_0a-B, m_0a+B] $. Thus,
setting $m=m_0$ in (\ref{equ1.5-1}) we obtain that
\begin{align*}
    f(x-2t_0+2m_0a)=\lambda_{m_0}\overline{g(-x+2t_0)},\quad \mbox{ a.e. } x\in [t_0-B, t_0].
\end{align*}
It follows from (\ref{equ1.19}) that
\[|f(x-2t_0)|=|f(x-2t_0+2m_0a)|=|g(-x+2t_0)|,\quad \mbox{ a.e. } x\in [t_0-B, t_0].
\]
Similarly, for any $x\in [t_0, t_0+B]$, we find that $-x+2t_0\in[t_0-B, t_0] \subset[(m_0-1)a-B, (m_0-1)a+B] $. Thus,
\[|f(x-2t_0)|=|g(-x+2t_0)|,\quad \mbox{ a.e. } x\in [t_0-B, t_0+B].
\]
Together with (\ref{equ1.15-1}), this yields
\begin{align}\label{equ1.17}
|f(x)|=|g(-x+2t_0)|=|f(x-2t_0)|,\quad \mbox{ a.e. } x\in [t_0-B, t_0+B].
\end{align}
For any $l\in\mathbb Z$,
we see from (\ref{equ1.19}) and (\ref{equ1.17}) that
\[
  |f(x)|=|f(x+2la)|=|f(x+2la-2t_0)|=|f(x-2t_0)|,\, \mbox{ a.e. } x\in [t_0-2la-B, t_0-2la+B].
\]
Hence
\begin{align}\label{equ1.18}
|f(x)|=|f(x-2t_0)|,\quad \mbox{ a.e. } x\in\mathbb{R}.
\end{align}

Since $t_0/a$ is irrational, it follows from Weyl's criterion that the sequence $\{ lt_0 \bmod a:l\in\mathbb{Z}\}$ is dense in $[0,a]$.
Hence there exists a subsequence $\{l_{k}\}_{k\ge 1}$ such that
$$\lim_{k \to \infty} l_{k}t_0 \bmod a =0.$$ By (\ref{equ1.18}) we observe that
\[
|f(x)|=|f(x+2l_{k}t_0)|=
|f(x+2(l_{k}t_0 \bmod a))|,\quad \mbox{ a.e. } x\in \mathbb{R}, \ \forall k\ge 1.
\]
Applying Lemma \ref{Lm:L1} yields that
\begin{align*}
|f(x)|=c_0,\quad \mbox{ a.e. } x\in \mathbb{R}
\end{align*}
for some constant $c_0$. Since $f$ is $(2B -a)$-nonseparable, we know that $c_0\ne 0$ and
\begin{align}\label{equ1.10}
|f(x)|=|g(x)|=c_0,\quad \mbox{ a.e. }
\end{align}
Let $h(x)=f(x)/g(x)$. Then by (\ref{eq:ec5}) and (\ref{eq:ec5-1}), we deduce that
\begin{align*}
h(x)=\frac{f(x)}{g(x)}=\frac{\lambda_{m-1}\overline{\lambda_{m}}f(x  +2 a)}{\lambda_{m-1}\overline{\lambda_{m}}f(x  +2 a)}=h(x+2a),
\end{align*}
for a.e. $x\in [(m-1)a-B,(m-2)a+B]$ and $m\in\mathbb{Z}$. Hence
\begin{align} \label{eq:ec10}
h(x)=h(x+2a),\quad \mbox{ a.e. } x\in \mathbb{R}.
\end{align}
On the other hand, by (\ref{equ1.5-1}) and (\ref{equ1.10}), we observe that
\begin{align*}
\begin{cases}
h(x)=\frac{f(x)}{g(x)}=\frac{\lambda_{m}\overline{g(2ma-x)}}
{\lambda_{m}\overline{f(2ma-x)}}
=\frac{ f(2ma-x)}{ g(2ma-x)}=h(2ma-x),&\quad \mbox{ a.e. } x\in [ma-B, ma+B],\\
h(x)=\frac{f(x)}{g(x)}
=\frac{\lambda(t_0)\overline{g(2t_0-x)}}
{\lambda(t_0)\overline{f(2t_0-x)}}
=\frac{ f(2t_0-x)}{ g(2t_0-x)}=h(2t_0-x),&\quad \mbox{ a.e. } x\in [t_0-B, t_0+B].
\end{cases}
\end{align*}
Since $h$ has a period $2a$ which is no greater than $2B$,
similarly to (\ref{equ1.18}) we obtain that
\begin{align*}
\begin{cases}
h(x)=h(2ma-x),&\quad \mbox{ a.e. } x\in \mathbb{R},\\
h(x)=h(2t_0-x),&\quad \mbox{ a.e. } x\in \mathbb{R}.
\end{cases}
\end{align*}
Letting $m=0$, we have
\begin{align*}
h(x)=h(2t_0-x)=h(x-2t_0),\quad \mbox{ a.e. } x\in \mathbb{R}.
\end{align*}
Since $t_0/a$ is irrational and $|h|=1$ a.e., similar to (\ref{equ1.10}) we obtain that  $h(x)=c_1$, a.e. on  $\mathbb{R}$ for some constant $c_1$. Hence $f\sim g$.

The negative part is proved in
Examples~\ref{ex:ex3}, \ref{ex:ex1} and Remark~\ref{rmk}. This completes the proof.
\end{proof}

\begin{example}\label{ex:ex3}
If $a> B$, then the measurements
(\ref{equ1.6}) fail to determine $f$ uniquely up to a global phase.
\end{example}

Let $f,g\in L_{loc}^p$ be such that
\begin{align*}
f(x) &=
  -f^*(x),
        \quad \forall x> a-B,\\
g(x) & = \begin{cases}
        f(x), & x< -(a-B), \\
        -f^*(x),   & -(a-B)\le x\le a-B, \\
        f(x), & x> a-B.
     \end{cases}
\end{align*}
When $m\le -1$ or $m\ge 1$, we have $f=g$ on $[ma-B,ma+B]$.
Hence for any $t\in(-\infty,-a]\cup [a,\infty)$,
\begin{equation}\label{eq:ed5}
  |\mathcal V_{\phi}f(t, \omega)|
   = |\mathcal V_{\phi}g(t, \omega)|,\quad \forall \omega\in\mathbb R.
\end{equation}
When $m=0$, we have
\[
  g(x) = -f^*(x),\quad x\in [-B,B].
\]
Hence (\ref{eq:ec3}) holds for all $m\in \mathbb{Z}$ and $t_0 \in (-\infty,-a)\cup (a,\infty)\setminus \mathbb{Q}$.
However, if $f(x)$ and $-f^*(x)$
are not identical on $[-(a-B),a-B]$,
then $f\not\sim g$.

\begin{example} \label{ex:ex1}
If $f$ is $(2B-a)$-separable, then the measurements
(\ref{equ1.6}) fail to determine $f$ uniquely up to a global phase.
\end{example}

Suppose that
$f(x)=0$ on $[-B,B-a]$ and $f$ is $\varepsilon$-nonseparable on $(-\infty,-B)$
and $(B-a,\infty)$ for some $\varepsilon>0$. Let
\[
g(x) =\begin{cases}
 f(x), & x\in(-\infty,B-a), \\
 -f(x), &  x \in(B-a,\infty).
\end{cases}
\]
 Then
 $f\not\sim g$.
 On the other hand, for any $t \in (-\infty,-a]\cup [0,\infty)$,
\begin{align*}
|\mathcal{V}_{\phi} f(t,\omega)|&=\left|\int_{-B}^{B}f(x+t) \overline{\phi(x)} e^{-2 \pi i x \omega} \rmd x\right|\\
&=\left|\int_{-B}^{B}g(x+ t) \overline{\phi(x)} e^{-2 \pi i x \omega} \rmd x\right|\\
&=|\mathcal{V}_{\phi} g(t,\omega)|
\end{align*}
Hence
(\ref{eq:ec3}) holds for all $m\in \mathbb{Z}$ and $t_0 \in (-\infty,-a)\cup (0,\infty)\setminus \mathbb{Q}$.

\begin{remark}\label{rmk}
If $t_0/a$ is a rational number
or
the measurements
$
\{|\mathcal{V}_{\!\phi} f(t_0,\omega_{n})|,
|\mathcal{V}_{\!\psi} f\!(t_0,\omega_{n})|\!\!:
n \ge 1\}$
are removed,
then the conclusion of Theorem \ref{thm1a} fails.
\end{remark}

Recall that $\phi(x) = \overline{\phi(-x)}$.
Let $f(x)=(2+\sin(\pi x/a))e^{\pi i x/(6a)}$ and $g(x)=(2-\sin(\pi x/a))e^{\pi i x/(6a)}$.
 We have
\begin{align*}
|\mathcal{V}_{\phi} f(ma,\omega_{n})|&=\left|\int_{-B}^{B}f(x+ma) \overline{\phi(x)} e^{-2 \pi i x \omega_{n}} \rmd x\right|\\
&=\left|\int_{-B}^{B}\left(2+(-1)^m \sin\big(\frac{\pi x}{a}\big)\right)e^{\pi i x/(6a)} \overline{\phi(x)} e^{-2 \pi i x \omega_{n}} \rmd x\right|\\
&=\left|\int_{-B}^{B}\left(2-(-1)^m \sin\big(\frac{\pi x}{a}\big)\right)e^{-\pi i x/(6a)} \overline{\phi(-x)} e^{2 \pi i x \omega_{n}} \rmd x\right|\\
&=\left|\int_{-B}^{B}\left(2-(-1)^m \sin\big(\frac{\pi x}{a}\big)\right)e^{-\pi i x/(6a)} \phi(x)
 e^{2 \pi i x \omega_{n}} \rmd x\right|\\
&=\left|\int_{-B}^{B}\left(2-(-1)^m \sin\big(\frac{\pi x}{a}\big)\right)e^{\pi i x/(6a)} \overline{\phi(x)} e^{-2 \pi i x \omega_{n}} \rmd x\right|\\
&=\left|\int_{-B}^{B}g(x+ma) \overline{\phi(x)} e^{-2 \pi i x \omega_{n}} \rmd x\right|\\
&=|\mathcal{V}_{\phi} g(ma,\omega_{n})|.
\end{align*}
Similar arguments show that $|\mathcal{V}_{\psi} f(ma,\omega_{n})|=|\mathcal{V}_{\psi} g(ma,\omega_{n})|$.
But $f$ and $g$ do not agree up to a global phase.

On the other hand, if
we replace $t_0$ by a finite set of points $\{t_1,...,t_k:t_j/a\in\mathbb{Q},~ j=1,...,k\}$, then the conclusion of Theorem \ref{thm1a} also fails. Indeed, we can choose some $b\in\mathbb{R}$ such that $a/b\in\mathbb{Z}$ and $t_j/b\in\mathbb{Z}$. Then the sequence in (\ref{equ1.6}) is a subset of $\{|\mathcal{V}_{\phi} f(mb,\omega_{n})|$, $|\mathcal{V}_{\psi} f(mb,\omega_{n})|:n\ge 1,
m \in \mathbb{Z}\}$.
It follows from previous arguments
that (\ref{equ1.6}) fails to determine
$f$ up to a global phase.

Next we give a proof of
Theorem~\ref{thm:main}.

\begin{proof}[Proof of Theorem~\ref{thm:main}]
From the proof of Theorem \ref{thm1a}, we observe the following.

If there exists some $m_0\in\mathbb{Z}$ such that (\ref{equ1.5}) holds,   then $f = \lambda_{m_0} g$, a.e.

If (\ref{equ1.5-1}) holds for all $m\in\mathbb{Z}$, then  (\ref{equ1.19}) holds. That is,
\begin{align*}
|f(x)|=|f(x+2a)|,\quad \mbox{ a.e. } x\in\mathbb{R}.
\end{align*}
Given that $f\in L^{p}(\mathbb R)$, we deduce that $f=0$, a.e., which contradicts the non-separability of $f$.

On the other hand, by letting $f\in L^p(\mathbb R)$ in
Example~\ref{ex:ex3}, we obtain a counterexample for
the case $a>B$.
And letting $f\in L^p(\mathbb R)$ in
Example~\ref{ex:ex1} yileds a counterexample
for $(2B-a)$-separable signals.
This completes the proof.
\end{proof}

\begin{remark}
We point out that with a slight change of the measurements,
Theorem~\ref{thm1a}
also works for the more general case:
$\phi\in L^1$ and $f\in L_{loc}^1$.
\end{remark}

In fact, we see from Lemma~\ref{Lm:L0}
that for almost $t\in\mathbb R$,
$f\cdot \phi(\cdot-t-ma)$ is integrable
for all $m\in\mathbb Z$. Fix such a $t_0$.
Then the measurements
\begin{align}\label{eq:ed3}
\{|\mathcal{V}_{\phi} f(t_0+ma,\omega_{n})|,~|\mathcal{V}_{\psi} f(t_0+ma,\omega_{n})|,~|\mathcal{V}_{\phi} f(t_1,\omega_{n})|,~|\mathcal{V}_{\psi} f(t_1,\omega_{n})|:n\ge 1, m \in \mathbb{Z}\}
\end{align}
determine $f$ up to a global phase
provided $(t_1-t_0)/a$ is an irrational number.


\end{document}